\documentclass[11pt]{article}
\usepackage{smile}

\usepackage{fullpage}
\usepackage{framed}
\usepackage{mdframed}
\usepackage{enumerate}
\usepackage[inline]{enumitem}
\usepackage[T1]{fontenc}
\usepackage{moresize}
\usepackage{bm}
\usepackage{dsfont}
\usepackage{amsmath}
\usepackage{amssymb}
\usepackage{amsthm}
\usepackage{amsfonts}
\usepackage{stmaryrd}
\usepackage{array}
\usepackage{mathrsfs}
\usepackage{mathtools} 
\usepackage{extarrows}
\usepackage{stackrel}
\usepackage[nodisplayskipstretch]{setspace}
\usepackage{color}
\usepackage[usenames,dvipsnames]{xcolor}
\usepackage{cancel}
\usepackage{soul}
\usepackage{xfrac}
\usepackage{siunitx}
\usepackage{graphicx}
\usepackage{float}
\usepackage{rotating}
\usepackage{subcaption}
\usepackage{overpic}
\usepackage[all]{xy}
\usepackage{tikz}
\usetikzlibrary{arrows,matrix,positioning,calc,automata,patterns}
\usepackage{booktabs}
\usepackage{dcolumn}
\usepackage{multirow}
\usepackage{diagbox}
\usepackage{lscape}
\usepackage{rotating}
\usepackage{verbatim}
\usepackage{listings}
\usepackage{algorithm}
\usepackage{algpseudocode}
\usepackage{fancyvrb}
\usepackage{hyperref}
\usepackage[round]{natbib}
\usepackage{sectsty}

\hypersetup{
    bookmarks=true,         
    unicode=false,          
    pdftoolbar=true,        
    pdfmenubar=true,        
    pdffitwindow=false,     
    pdfstartview={FitH},    
    pdftitle={My title},    
    pdfauthor={Author},     
    pdfsubject={Subject},   
    pdfcreator={Creator},   
    pdfproducer={Producer}, 
    pdfkeywords={key1, key2}, 
    pdfnewwindow=true,      
    colorlinks=true,        
    linkcolor=blue,         
    citecolor=blue,         
    filecolor=blue,         
    urlcolor=cyan           
}

\def\given{\,|\,}

\def\tr{\mathop{\text{tr}}\kern.2ex}

\def\E{{\mathbb E}}
\def\R{{\mathbb R}}

\newcolumntype{L}[1]{>{\raggedright\let\newline\\\arraybackslash\hspace{0pt}}m{#1}}
\newcolumntype{C}[1]{>{  \centering\let\newline\\\arraybackslash\hspace{0pt}}m{#1}}
\newcolumntype{R}[1]{>{ \raggedleft\let\newline\\\arraybackslash\hspace{0pt}}m{#1}}
\newcolumntype{d}[1]{D{.}{.}{#1}}
\newcolumntype{H}{>{\setbox0=\hbox\bgroup}c<{\egroup}@{}}
\newcolumntype{Z}{>{\setbox0=\hbox\bgroup}c<{\egroup}@{\hspace*{-\tabcolsep}}}
\renewcommand{\d}{\textrm{\sf d}}

\newcommand{\pnorm}[2]{\lVert#1\rVert_{#2}}

\numberwithin{equation}{section}
\newtheorem{theorem}{Theorem}[section]
\newtheorem{lemma}{Lemma}[section]
\newtheorem{proposition}{Proposition}[section]

\newtheorem{innercustomgeneric}{\customgenericname}

\providecommand{\customgenericname}{}
\newcommand{\newcustomtheorem}[2]{%
  \newenvironment{#1}[1]
  {%
   \renewcommand\customgenericname{#2}%
   \renewcommand\theinnercustomgeneric{##1}%
   \innercustomgeneric
  }
  {\endinnercustomgeneric}
}
\newcustomtheorem{customtheorem}{Theorem}
\newcustomtheorem{customassumption}{Assumption}
\newcustomtheorem{customlemma}{Lemma}
\newcustomtheorem{customexample}{Example}
\theoremstyle{definition}

\newtheorem{remark}{Remark}[section]

\begin{document}

\setlength{\abovedisplayskip}{5pt}
\setlength{\belowdisplayskip}{5pt}
\setlength{\abovedisplayshortskip}{5pt}
\setlength{\belowdisplayshortskip}{5pt}
\hypersetup{colorlinks,breaklinks,urlcolor=blue,linkcolor=blue}

\title{\LARGE Nonparametric mixture MLEs under Gaussian-smoothed optimal transport distance}

\author{
Fang Han\thanks{Department of Statistics, University of Washington, Seattle, WA 98195, USA; e-mail: {\tt fanghan@uw.edu}}, ~~Zhen Miao\thanks{Department of Statistics, University of Washington, Seattle, WA 98195, USA; e-mail: {\tt zhenm@uw.edu}}, ~and~Yandi Shen\thanks{Department of Statistics, University of Chicago, Chicago, IL 60637, USA. E-mail: \tt{ydshen@uchicago.edu}}}

\date{\today}

\maketitle


\begin{abstract} 
The Gaussian-smoothed optimal transport (GOT) framework, pioneered in \cite{goldfeld2020convergence} and followed up by a series of subsequent papers, has quickly caught attention among researchers in statistics, machine learning, information theory, and related fields. One key observation made therein is that, by adapting to the GOT framework instead of its unsmoothed counterpart, the curse of dimensionality for using the empirical measure to approximate the true data generating distribution can be lifted. The current paper shows that a related observation applies to the estimation of nonparametric mixing distributions in discrete exponential family models, where under the GOT cost the estimation accuracy of the nonparametric MLE can be accelerated to a polynomial rate. This is in sharp contrast to the classical sub-polynomial rates based on unsmoothed metrics, which cannot be improved from an information-theoretical perspective.  A key step in our analysis is the establishment of a new Jackson-type approximation bound of Gaussian-convoluted Lipschitz functions. This insight bridges existing techniques of analyzing the nonparametric MLEs and the new GOT framework.
\end{abstract}


\section{Introduction} \label{sec:intro}

Let $f(x\given\theta)$ be a known parametric density function with respect to a certain (counting or continuous) measure and $X_1, \ldots,X_n$ be $n$ i.i.d. observations drawn from the following {\it mixture density function},
\begin{align}\label{eq:model}
h_Q(x) := \int f(x\given\theta){\sf d}Q(\theta),
\end{align}
where $Q$ is unspecified and termed the {\it mixing distribution} of $\theta$.  Our goal is to estimate the unknown $Q$ based on $X_1,\ldots,X_n$. This is the celebrated nonparametric mixing distribution estimation problem, which has been extensively studied in literature \citep{lindsay1995mixture}. The focus of this paper is on studying the estimation of $Q$ in the case of (identifiable) \textit{discrete exponential family models} \citep{zhang1995estimating}, i.e., $f(x\given\theta)$ taking the following form:
\begin{align}\label{eq:exponential}
f(x \given \theta)=g(\theta)w(x)\theta^x,~~~{\rm with}~&x=0,1,2,\ldots, \text{ and } \notag\\
& w(x)>0~~~\text{for all }x\geq 0,\notag\\
&0\leq \theta\leq \text{ (a known fixed constant) }\theta_*< \min\{\theta_r, \infty\},
\end{align}
where $\theta_r\in(0,\infty]$ is the radius of convergence of the power series $\theta \mapsto \sum_{x=0}^{\infty}w(x)\theta^x$ and $g(\cdot)$ is analytic in a neighborhood of $0$. This model includes, among many others, Poisson and negative binomial distributions. 

Estimation of $Q$ under the discrete exponential family models has been extensively  investigated in literature through, e.g., the use of nonparametric maximum likelihood estimators (MLEs) \citep{simar1976maximum}, method of moments \citep{tucker1963estimate}, Fourier and kernel methods  \citep{zhang1995estimating, loh1996global, loh1997estimating}, and projection methods \citep{walter1991bayes,hengartner1997adaptive,roueff2005nonparametric}. Of particular interest to us is the MLE-based approach, partly due to its asymptotic efficiency under regular parametric models. In the case of nonparametric mixture models, the MLE can be written as 
\begin{align}\label{eq:MLE}
\hat Q := \argmax_{Q \text{ on }[0,\theta^*]}\sum_{i=1}^n \log h_Q(X_i),
\end{align}
which is a convex problem with efficient solving algorithms \citep{simar1976maximum}.

Although a proof of the consistency of $\hat Q$ has been standard now (cf. \cite{chen2017consistency}), of central importance to statisticians and machine learning scientists in making inference based on $\hat Q$ is its rate of convergence. In this regard,  \cite{zhang1995estimating} established the first minimax lower bound, indicating that, at the worst case, it is impossible for the MLEs to achieve a polynomial rate if measured using regular metrics such as the total variation distance and the optimal transport distance (OT; in this paper restricted to the Wasserstein-1 distance $W_1$). It is now known that, for Poisson mixtures, the minimax rate of convergence under OT distance is $\log\log n/\log n$ and could indeed be achieved by MLEs \citep[Theorems 6.1 and 6.2]{miao2021fisher}. This slow rate demonstrates that the estimation of $Q$ suffers severely from its nonparametric structure.

Interestingly, a similar fundamental ``curse'' also exists in using the empirical measure $P_n$ of an independent and identically distributed (i.i.d.) sample of size $n$ to approximate the true data generating distribution $P$ in $\mathbb{R}^d$, for which the minimax rate under the cost of OT was shown to be $n^{-1/d}$ as $d>2$ \citep[Theorem 1]{fournier2015rate}. Partly motivated by a problem of estimating information flows in deep neural networks \citep{goldfeld2018estimating}, \cite{goldfeld2020convergence} introduced a new measure of distance $W_1^{\sigma}$, which is termed the {\it Gaussian-smoothed OT (GOT)} distance. The GOT distance, like the unsmoothed OT one, is a metric on the probability measure space with finite first moment that metrizes the weak topology. In addition, both $W_1^\sigma$ and the corresponding optimal transport plan converge weakly to the corresponding unsmoothed versions as the smoothing parameter $\sigma\rightarrow 0$ (cf. \citet[Theorems 2, 3, and 4]{goldfeld2020gaussian}).

Under this new distance and with some further moment conditions on $P$, Goldfeld et al. was able to prove an upper bound of $W_1^\sigma(P_n,P)$ that is of the best possible root-$n$ order and thus overcomes the curse of dimensionality faced with the classical unsmoothed scenario. Subsequent developments establish the weak convergence of $W_1^\sigma(P_n,P)$ to a functional of a Gaussian process \citep{sadhu2021limit}, weaken the moment assumption \citep{zhang2021convergence},  and study high noise limit as $\sigma\to\infty$ \citep{chen2021asymptotics}.

One of the main contributions of this paper is to prove that an observation similar to what was made in \cite{goldfeld2020convergence} occurs to the nonparametric mixture MLEs, i.e., under some conditions on $w(\cdot)$, we have 
\begin{align}\label{eq:main}
\sup_{Q \text{ on }[0,\theta^*]}\E W_1^{\sigma}(\hat Q, Q) \leq C(\sigma,\theta_*,w) n^{-\eta(\theta_*,w)},
\end{align}
where $C$ and $\eta$ are two positive constants only depending on $\{\sigma,  \theta_*,w\}$ and $\{\theta_*,w\}$, respectively. Our result thus bridges two distinct areas, namely nonparametric mixing distribution estimation and empirical approximation to population distribution; in the earlier case, GOT is shown to boost the convergence rate to polynomial, while in the latter case GOT overcomes the curse of dimensionality. 

The main technical step in our proof of \eqref{eq:main} is a new Jackson-type bound on the error of degree-$k$ polynomial (for an arbitrary positive integer $k$) approximation to Gaussian-convoluted Lipschitz functions with a bounded support. Our result thus extends the classic Jackson's Theorem (\cite{jackson1921general}; see Lemma \ref{lem:jackson}) and paves a way to leverage existing technical tools of analyzing the nonparametric MLEs, devised in an early draft written by some of the authors in this paper \citep[Section 6]{miao2021fisher}.

\vspace{0.2cm}
\noindent {\bf Notation.} For any positive integer $n$, let $[n]:=\{1,\ldots,n\}$. For any two distributions $Q_1$ and $Q_2$ over $\mathbb{R}^d$, let $Q_1\ast Q_2$ represent the convolution of  $Q_1$ and $Q_2$, i.e., $Q_1\ast Q_2(A)=\int\int \ind_A(x+y){\sf d}Q_1(x){\sf d}Q_2(y)$, with $\ind_\cdot(\cdot)$ standing for the indicator function. For any two measurable functions $f,g$ on $\mathbb{R}^d$, $f\ast g$ represents their convolution, i.e., $f\ast g(x)=\int f(x-y)g(y){\sf d}y$. For any function $f:\mathbb{R}\to\mathbb{R}$ and $\alpha>0$, let $f^{(\alpha)}$ represent the $\alpha$-time derivative of $f$. The OT (i.e., Wasserstein $W_1$) distance between $Q_1$ and $Q_2$ is defined as 
\[
W_1(Q_1,Q_2):=\sup_{\ell\in {\rm Lip}_1}\int \ell (\d Q_1- \d Q_2),~~~\text{(Kantorovich-Rubinstein formula)}
\]
where the supremum is over all $1$-Lipschitz functions (under the Euclidean metric $\|\cdot\|$) on $\mathbb{R}^d$. Let $\mathcal{N}_\sigma$ represent the Gaussian distribution with mean $0$ and covariance matrix $\sigma^2I_d$, where $I_d$ stands for the $d$-dimensional identity matrix. Let $\phi_\sigma$ denote the corresponding density function. The GOT distance $W_1^{\sigma}$ is defined as
\[
W_1^{\sigma}(Q_1,Q_2):= W_1(Q_1\ast \mathcal{N}_\sigma, Q_2 \ast \mathcal{N}_\sigma).
\]
Let $\cP(\mathbb{R}^d)$ represent the set of all Borel probability measures on $\mathbb{R}^d$ and $\cP_1(\mathbb{R}^d)$ be the subset of $\cP(\mathbb{R}^d)$ with elements of finite first moment. Throughout the paper, $C, C',C'', c,c'$ are used to represent generic positive constants whole values may change in different locations.

\vspace{0.2cm}
\noindent {\bf Paper organization.} The rest of this paper is organized as follows. Section \ref{sec:prelim} gives the preliminaries on the studied nonparametric mixture models and the MLEs. Section \ref{sec:main} delivers the main results, including the key technical insights to the proof. Section \ref{sec:proof} collects the main proofs, with auxiliary proofs relegated to Section \ref{sec:auxiliary}.

\section{Preliminaries} \label{sec:prelim} 

\subsection{Nonparametric mixture MLEs}

Estimating the mixing distribution is known to be statistically challenging in a variety of nonparametric mixture models including the Gaussian \citep{wu2020optimal}, binomial \citep{tian2017learning,vinayak2019maximum}, and Poisson \citep{miao2021fisher} ones. Specific to the discrete exponential family models in the form of (\ref{eq:exponential}), the following $\log n$-scale information-theoretical lower bound formalized this difficulty under the standard unsmoothed $W_1$ distance.

\begin{theorem}[Minimax lower bounds under $W_1$ distance]\label{thm:lb} Let $\{X_i,i\in[n]\}$ be a random sample generated from the mixture density function $h_Q$ defined in \eqref{eq:model} and $n\geq 2$. We then have
\begin{itemize}
\item[(a)] For any $f(x\given \theta)$ taking the form \eqref{eq:exponential}, we have
\begin{align}\label{eq:lb1}
\inf_{\tilde Q}\sup_{Q~{\rm on }~[0,\theta_*]}\E W_1(Q,\tilde Q)\geq \frac{c}{\log n},
\end{align}
where the infimum is taken over all measurable estimators of the mixing distribution $Q$ with support on $[0,\theta_*]$ and $c=c(\theta_*)>0$ is a constant only depending on $\theta_*$.
\item[(b)] \citep[Theorem 6.2]{miao2021fisher} Suppose further $f(x\given \theta)=e^{-\theta}\theta^x/x!$ to be the probability mass function of the Poisson with natural parameter $\theta$. We can further tighten the lower bound  in \eqref{eq:lb1} to be
\[
\inf_{\tilde Q}\sup_{Q~{\rm on }~[0,\theta_*]}\E W_1(Q,\tilde Q)\geq \frac{c'\log\log n}{\log n},
\]
for a constant $c'=c'(\theta_*)>0$ depending only on $\theta_*$.
\end{itemize}
\end{theorem}

\begin{remark}\label{remark:iden}
In \eqref{eq:exponential}, the condition that ``$w(x)>0$ for all nonnegative integer $x$'' is a (simplified) identifiability condition. Similar conditions were also posed in, e.g., \citet[Corollary 1]{zhang1995estimating} and \citet[Corollary 1]{loh1996global}. As a matter of fact, \citet[Theorem 1(a)]{stoyanov2011mixtures} showed that, if there exists a constant $C>0$ such that $f(x\given \theta)=0$ for all $x\geq C$ and $\theta\in[0,\theta_*]$, then $Q$ is not identifiable, i.e., there exist at least two distinct mixing distributions $Q_1,Q_2$ over $[0,\theta_*]$ such that $h_{Q_1}=h_{Q_2}$. On the other hand, it is straightforward to generalize the above result to the case of ``$w(x)>0$ for all $x\geq x_0$ for some nonnegative integer $x_0$ that is known to us''.
\end{remark}

In the past several decades, methods that provably (nearly) achieve the above minimax lower bound have been proposed; cf.  \cite{zhang1995estimating}, \cite{hengartner1997adaptive}, and \cite{roueff2005nonparametric} among many others. However, none of the above methods is likelihood-based, partly due to the theoretical challenges faced with analyzing the nonparametric MLEs. A major breakthrough towards understanding the rate of convergence of nonparametric mixture MLEs was made in \cite{vinayak2019maximum} for the binomial case and \cite{miao2021fisher} for the Poisson case. 

The following theorem provides an extension of \citet[Theorem 6.1]{miao2021fisher} to cover those mixture models of the general form \eqref{eq:exponential}.

\begin{theorem}[Minimax upper bounds of MLEs under $W_1$ distance]\label{thm:ub-W1} Let $\{X_i,i\in[n]\}$ be a random sample generated from the mixture density function $h_Q$ defined in \eqref{eq:model}, and $\hat Q$ be the MLE introduced in \eqref{eq:MLE}. The following are then true.
\begin{itemize}
\item[(a)] If there exists a universal constant $C\geq 1$ such that $1/w(x)\leq C^x$ for all $x \geq 1$, then there exist some $n'=n'(\theta_*,C)$ and $C'=C'(\theta_*,C)$ such that
\[
\sup_{Q ~{\rm on }~[0,\theta_*]}\E W_1(Q,\hat Q)\leq \frac{C'}{\log n}~~~\text{ for all }n\geq n'.
\]
\item[(b)] If there exists a universal constant $C\geq 1$ such that $1/w(x)\leq (Cx)^{Cx}$ for all $x\geq 1$, then there exists $n'=n'(\theta_*,C)$ and $C'=C'(\theta_*,C)$ such that
\[
\sup_{Q ~{\rm on }~[0,\theta_*]}\E W_1(Q,\hat Q)\leq \frac{C'\log\log n}{\log n}~~~\text{ for all }n\geq n'.
\]
\end{itemize}
\end{theorem}

\begin{remark}
The conditions enforced for $w(x)$ in Theorem \ref{thm:ub-W1} are classic and related to the identifiability of $Q$ discussed in Remark \ref{remark:iden}. Similar conditions were posed in \citet[Theorem 4]{zhang1995estimating}, \citet[Corollary 1]{loh1996global}, \citet[Theorem 1]{loh1997estimating}, and \citet[Corollary 1]{roueff2005nonparametric}.
\end{remark}

\begin{remark}
It is straightforward to verify that, after some standard operations including location shift, point mass inflation, and reparametrization, Theorem \ref{thm:ub-W1}(a) applies to, e.g., the (zero-inflated or $C$-truncated) negative binomial, the logarithmic  \citep{noack1950class}, the lost games \citep{gupta1984estimating}, as well as the generalized Poisson, negative binomial, and logarithmic \citep{janardan1982new} distributions; Theorem \ref{thm:ub-W1}(b) applies to, e.g., the (zero-inflated or $C$-truncated) Poisson as well as the Poisson polynomial \citep{cameron2013regression} distributions.
\end{remark}

\begin{remark}\label{remark:UB}
It may be of some theoretical interest to note that Theorem \ref{thm:ub-W1} can be further generalized to cover the following two cases.
\begin{itemize}
\item[(i)] If the following bound holds,
\[
1/w(x)\leq \underbrace{\exp\circ\exp\circ\cdots\exp}_{L}(Cx) \text{ for all } x, 
\]
then there exists a constant $C'=C'(C,\theta_*)$ such that 
\[
\sup_{Q\text{ on }[0,\theta_*]}\E W_1(Q,\hat Q) \leq C'/\underbrace{\log\circ \log\circ\cdots\circ \log}_L(n)=:C'/\log_L(n)~~\text{ for all }n\geq n',
\]
where $n'=n'(C,\theta_*)$ is a sufficiently large integer. 
\item[(ii)] If the following bound holds,
\[
1/w(x)\leq (Cx)^{\cdots (Cx)} ~~~(L\text{ times power})~~~\text{ for all }x,
\]
then there exists a constant $C'=C'(C,\theta_*)$ such that 
\[
\sup_{Q\text{ on }[0,\theta_*]}\E W_1(Q,\hat Q) \leq C'\log_{L-1}(n)/\log_L(n)~~\text{ for all }n\geq n',
\]
where $n'=n'(C,\theta_*)$ is a sufficiently large integer. 
\end{itemize}
\end{remark}


\subsection{The GOT distance}

Theorem \ref{thm:lb} suggests that, under the $W_1$ cost, the sub-polynomial rate in estimating the mixing distribution of a nonparametric mixture model is information-theoretically optimal. As a matter of fact, the conclusion of Theorem \ref{thm:lb} goes beyond the discrete exponential family models studied in this paper; cf. \citet[Proposition 8]{wu2020optimal} for a similar phenomenon in the nonparametric Gaussian mixture models. 

Revising the Wasserstein distance through convolution/smoothing has a long and rich history. In probability theory, this is interestingly related to heat semigroup operators on Riemannian manifold \citep{von2005transport}, which reveals its connection to the Ricci curvature. More recently, stemming from the interest in estimating the mutual information of deep networks, \cite{goldfeld2018estimating} initiated the study of GOT distances, introduced as a smoothed alternative to the classic OT metric.

Indeed, the GOT distance is now known to be able to effectively alleviate some undesired issues associated with the OT distances. Let us start with the following simple fact, that the $W_1$-distance is non-increasing under convolution.

\begin{lemma}\label{lem:simple} Consider $\mu_1,\mu_2,\nu\in \cP_1(\mathbb{R}^d)$ be arbitrary three Borel probability measures on $\mathbb{R}^d$ with finite first moment. We then have
\[
W_1(\mu_1\ast \nu, \mu_2\ast \nu) \leq W_1(\mu_1,\mu_2).
\]
\end{lemma}

Lemma \ref{lem:simple} confirms that the GOT distance is no greater than the original OT distance, but it does not quantify the difference. For that purpose, the existing literature has provided us with an interesting example, i.e., in approximating the population measure using the empirical one. In detail, suppose $P_n$ is the empirical measure of $P$, and both are supported on $\mathbb{R}^d$ with some integer $d\geq 3$. \citet[Theorem 1]{fournier2015rate} showed that
\[
\sup_{P: \E_PX^2<\infty}\E W_1(P,P_n) \asymp n^{-1/d},
\]
which is faced with severe curse of dimensionality as the dimension $d$ becomes larger. In a recent paper of \cite{goldfeld2020convergence}, the authors showed that, via appealing to the GOT one, this curse can be effectively handled. More specifically, they proved that, as long as $P$ is sub-gaussian with a fixed subgaussian constant, we have
\[
\E W_1^{\sigma}(P,P_n) \lesssim n^{-1/2},
\]
which is the parametric rate of convergence. See also \cite{sadhu2021limit}  for the limiting distribution of $\sqrt{n}W_1^{\sigma}(P,P_n)$ as well as \cite{zhang2021convergence} for the relaxation of the moment conditions on $P$. 

The purpose of this paper is to present the second and also a statistically interesting example, for which adopting the GOT distance can significantly accelerate the convergence rate of a statistical procedure. 

\section{Main results}\label{sec:main}

The following theorem is the main result of this paper.

\begin{theorem}\label{thm:main}
Let $\{X_i,i\in[n]\}$ be a random sample generated from the mixture density function $h_Q$ defined in \eqref{eq:model}, and $\hat Q$ be the MLE introduced in \eqref{eq:MLE}. Suppose that there exist some universal positive constants $c_1,c_2,c_3,C_1,C_2,C_3$ such that one of the following two conditions holds, 
\begin{itemize}
\item[(i)] $1/w(x)\leq C_1C_2^x$ for all $x=1,2,\ldots$;
\item[(ii)] $c_1c_2^xx^{c_3x}\leq 1/w(x)\leq C_1C_2^xx^{C_3x}$ for all $x=1,2,\ldots$.
\end{itemize}
Then we have
\[
\sup_{Q \text{ on }[0,\theta_*]}\E W_1^{\sigma}(Q,\hat Q)\leq C\cdot n^{-c}.
\]
Here $C=C(\sigma,\theta_*,c_1,c_2,c_3,C_1,C_2,C_3)$ and $c=c(\theta_*,c_3,C_2,C_3)$ are two positive constants.
\end{theorem} 

\begin{remark}
Let us point out some results in the nonparametric mixture model literature that are relevant to ours. \cite{lambert1984asymptotic} studied the convergence of $h_{\hat Q}$ to $h_Q$ in a specific nonparametric Poisson mixture model based on the regular unsmoothed distance. They observed that the convergence rate can be nearly parametric; cf. Proposition 3.1 therein. This observation is particularly relevant to ours as the $h_{\cdot}$ operation is intrinsically also ``smoothing'' the probability measure. A similar observation was made in \cite{wu2020optimal}, who studied approximating the mixing distributions in Gaussian mixture models via moment matching. In particular, their Lemma 8 considers bounding the Gaussian-smoothed chi-squared distance between two subgaussian distributions whose first $k$ moments are matched. Their bound suggests a similar exponential-order improvement as ours. However, it is clear from the context that the proof techniques in \cite{lambert1984asymptotic} and \cite{wu2020optimal} are distinct from the current paper, where, as we detail next, the conclusion is arrived via a new Jackson-type bound.
\end{remark}

\begin{remark}
In Theorem \ref{thm:main} the explicit value of $c$ was not exposed. For readers of interest, considering $\epsilon\in(0,1)$ to be an arbitrarily small positive constant, the largest possible $c$ we can obtain for the Poisson mixture is 
\[
1/10-\epsilon 
\]
and for negative binomial mixture is
\[
\Big[2\Big\{1+2\cdot \frac{\log(e/\theta_*)}{\log(1/\theta_*)}\Big\}\Big]^{-1}-\epsilon, \text{ recalling that }\theta_*\in (0,1) \text{ in this case}.
\]
Although it is certainly not the main interest of this paper to devise the sharpest possible value of $c$, it is our conjecture that for any fixed $\sigma$, (at the worst case) $c$ would have to be strictly smaller than $1/2$. In other words,  the parametric root-$n$ rate as was observed in \cite{goldfeld2020convergence} cannot be recovered in the setting of nonparametric mixture MLEs considered in this paper. A detailed investigation of the lower bound of $\E W_1^{\sigma}(Q,\hat Q)$ is beyond the scope of this paper, but will be the topic of a subsequent future work.
\end{remark}

Next we give a proof sketch of Theorem \ref{thm:ub-W1}. Invoking the same trick that was used in the proof of Theorem 6.1(a) in \cite{miao2021fisher}, for any given 1-Lipschitz function $\ell(\cdot)$ such that $\ell(0)=0$, we introduce the following function to approximate it,
\[
\hat\ell_k(\theta):=\sum_{x=0}^{k}b_{x,\ell} f(x|\theta),~~{\rm for}~b_{x,\ell}\in\mathbb{R} \text{ and }\theta\in[0,\theta_*].
\]
Some straightforward manipulations (see Section \ref{sec:ub-W1} for details) then yield
\begin{align}\label{eq:key2}
W_1(Q,\hat Q)\leq \sup_{\ell\in{\rm Lip}_1, \ell(0)=0}\Big\{2\sup_{\theta\in[0,\theta_*]}\Big|\ell(\theta)-\hat\ell_k(\theta)\Big|+\sum_{x=0}^{k}b_{x,\ell}&\Big(h_Q(x)-h_Q^{\rm obs}(x)\Big) \notag\\
&+\sum_{x=0}^{k}b_{x,\ell}\Big(h_Q^{\rm obs}(x)-h_{\hat Q}(x)\Big) \Big\},
\end{align}
where $h_Q^{\rm obs}(x):=n^{-1}\sum_{i=1}^n\ind(x=X_i)$. The sub-polynomial bound of $\E W_1(Q,\hat Q)$ could then be explained by the following fact (detailed proofs to be presented in Section \ref{sec:ub-W1}): For any function $\ell$ considered above and any $k=1,2,\ldots$, there exists an $\hat\ell_k$ and a constant $C=C(\theta_*)>0$ only depending on $\theta_*$ such that the following two inequalities hold. First, a Jackson-type bound (see Lemma \ref{Prop:BoundOnf1}): 
\begin{align}\label{eq:key1}
\sup_{\ell\in{\rm Lip}_1, \ell(0)=0}\sup_{\theta\in[0,\theta_*]}\Big| \ell(\theta)-\hat\ell_k(\theta)\Big| \leq C/k;
\end{align}
second, 
\[
\sup_{\ell\in{\rm Lip}_1, \ell(0)=0}\max_{0\leq x\leq k}\Big|b_{x,\ell}\Big| \leq C^k\max_{1\leq x\leq k}\Big\{\frac{1}{w(x)}\Big\}.
\]
Plugging different bounds of $\max\{1/w(x)\}$ into the above two inequalities then gives us the desired results in Theorem \ref{thm:ub-W1}. 

With these concepts in mind, let us then move on to examine the case when the GOT distance is used. Similar to the derivation of \eqref{eq:key2} and further noting that $\int \ell {\sf d}(Q \ast \mathcal{N}_\sigma)=\int \ell_{\sigma}{\sf d}Q$ with 
\[
\ell_{\sigma}:=\ell\ast \phi_\sigma, 
\]
one can show that
\begin{align}\label{eq:key3}
W_1^\sigma(Q,\hat Q)\leq \sup_{\ell\in{\rm Lip}_1, \ell(0)=0}\Big\{2\sup_{\theta\in[0,\theta_*]}\Big|\ell_\sigma(\theta)-\ell_\sigma(0)-\hat\ell_k(\theta)\Big|+&\sum_{x=0}^{k}b_{x,\ell}\Big(h_Q(x)-h_Q^{\rm obs}(x)\Big) \notag\\
&+\sum_{x=0}^{k}b_{x,\ell}\Big(h_Q^{\rm obs}(x)-h_{\hat Q}(x)\Big) \Big\}.
\end{align}
The last two terms in \eqref{eq:key3} can be similarly handled as in \eqref{eq:key2}, and it remains to control the first term. To this end, we introduce the following lemma, which turns out to be an extension of the Jackson-type one.

\begin{lemma}[Polynomial approximation of Gaussian-convoluted Lipschitz functions]\label{lem:key} Let $0\in[a,b]\subset \mathbb{R}$ be a bounded interval and let $\ell(\cdot)$ be a 1-Lipschitz function over $[a,b]$. For any $\sigma>0$ and integer $k>1$, there exist a constant $C=C(a,b)>0$ only depending on $a,b$ and a polynomial $p_k(\cdot)$ of degree at most $k$ such that
\[
\sup_{\theta\in[a,b]}\Big|\ell_\sigma(\theta)-\ell_\sigma(0)-p_k(\theta)  \Big| \leq C_1e\sigma\cdot  \Big[\frac{2\sqrt{e}\sigma \sqrt{k}}{b-a}\Big]^{-k}k^{-1/4},
\]
where we recall that $\ell_\sigma:=\ell\ast\phi_\sigma$ with $\phi_\sigma$ standing for the density function of $\mathcal{N}_\sigma$.
\end{lemma}

In striking contrast to the linear convergence in Jackson-type bounds such as \eqref{eq:key1}, Lemma \ref{lem:key} states that approximation to Gaussian-convoluted Lipschitz functions by degree-$k$ polynomials is super-exponentially fast, hinting a substantial gain of convergence speed whence GOT distances are used to quantify the distance. We refer to Section \ref{subsec:proof_main} for the complete proof of Theorem \ref{thm:main}.

\section{Proofs}\label{sec:proof}

In the subsequent proofs, we sometimes drop the track of dependence on $C,C'$ for simplicity.

\subsection{Proof of Theorem \ref{thm:lb}}

\begin{proof}
The proof is based on Le Cam's two-point method (cf. \citet[Chapter 2.3]{tsybakov2009introduction}) and uses the following proposition.
\begin{proposition}[Lemma 3 in \cite{tian2017learning}, Proposition 4.3 in \cite{vinayak2019maximum}.]\label{prop:kong} For any positive integer $k$ and any $M>0$, there exist two distributions $P_1,P_2$ with support in $[0,M]$  such that $P_1,P_2$ have first $k$ moments identical and $W_1(P_1,P_2)\geq M/(2k)$.
\end{proposition}

We first upper bound $g^{(x)}(0)\theta_*^x/x!$. To this end, define $\tilde g(\theta):=g(\theta_*\theta)$. We then have 
\[
\frac{1}{\tilde{g}(\theta)} = \sum_{x=0}^{\infty}w(x)\theta_*^x\theta^x~~\text{ for all }\theta\in [0,1]
\]
by the definition of the mixing density function in \eqref{eq:exponential}. Furthermore, the radius of convergence of $\sum_{x=0}^{\infty}w(x)\theta_*^x\theta^x$ is $\theta_r/\theta_*>1$. Accordingly, by \citet[Corollary 1.1.10]{krantz2002primer}, there exists some universal constant $C>0$ such that
\[
w(x)\theta_*^x\leq C~~\text{ for all }x=0,1,2,\ldots.
\]
The proof of \citet[Corollary 1.1.12]{krantz2002primer} then yields that $\tilde g(\theta)$, of the form $\tilde g(\theta)=\sum_{x=0}^{\infty}\tilde g^{(x)}(0)\theta^x/x!$, has a radius of convergence at least $1/(C+1)$. Invoking \citet[Corollary 1.1.10]{krantz2002primer} again shows that there exists another universal constant $C'>0$ such that
\begin{align}\label{eq:proof1}
|g^{(x)}(0)\theta_*^x/x!| = |\tilde g^{(x)}(0)/x!|\leq C'(C+1)^x, ~~\text{ for all }x=0,1,2,\ldots.
\end{align}

We then combine \eqref{eq:proof1} with Proposition \ref{prop:kong} to finish the proof. On one hand, for any $k=1,2,\ldots$, Proposition \ref{prop:kong} guarantees the existence of two distributions $Q_1,Q_2$ over $[0,\theta_*/(C+3)]$ such that 
\begin{align*}
\int \theta^x{\sf d}Q_1(\theta) = \int \theta^x{\sf d}Q_2(\theta), ~~~\text {for all }x\in[k]~~~\text{ and }~~~W_1(Q_1,Q_2)\geq \theta_*/(2(C+3)k).
\end{align*}
On the other hand, the total variance distance between $h_{Q_1}$ and $h_{Q_2}$ satisfies
\begin{align*}
{\rm TV}(h_{Q_1},h_{Q_2}) &\leq \frac{1}{2}\sum_{x=0}^{\infty}\Big|\int_0^{\theta_*}g(\theta)w(x)\theta^x{\sf d}Q_1(\theta)- \int_0^{\theta_*}g(\theta)w(x)\theta^x{\sf d}Q_2(\theta) \Big|\\
&\leq \sum_{x=0}^{\infty}w(x)\sum_{m:m+x\geq k+1}\frac{|g^{(m)}(0)|}{m!}\Big(\frac{\theta_*}{C+3}\Big)^{m+x}\\
&\leq 2(C+3)^2C'\Big(\frac{C+2}{C+3}\Big)^k.
\end{align*}
Picking $k=k(n)$ so that 
\[
2(C+3)^2C'\Big(\frac{C+2}{C+3}\Big)^k = 1/(2n),
\]
it follows from Le Cam's lower bound for two hypotheses that, denoting $Q^{\otimes n}$ to be the $n$-time product measure of $Q$, 
\begin{align*}
\inf_{\tilde Q}\sup_{Q}\E W_1(Q,\tilde Q)&\geq \frac{1}{2}W_1(Q_1,Q_2)\Big\{1-{\rm TV}(h_{Q_1}^{\otimes n}, h_{Q_2}^{\otimes n})\Big\}\\
&\geq \frac{1}{2}W_1(Q_1,Q_2)\{1-n/(2n)\}=\frac{1}{4}W_1(Q_1,Q_2),
\end{align*}
with $W_1(Q_1,Q_2)\geq \theta_*/(2(C+3)k)$ by the construction. This completes the proof.
\end{proof}

\subsection{Proof of Theorem \ref{thm:ub-W1}}\label{sec:ub-W1}

\begin{proof}[Proof of Theorem~\ref{thm:ub-W1}]
By definition of $W_1$, we have
\begin{align*}
W_1(Q_1,Q_2) = \sup_{\ell \in \text{Lip}_1}\int \ell (\d Q_1 - \d Q_2) = \sup_{\ell \in\text{Lip}_1, \ell(0) = 0}\int \ell (\d Q_1 - \d Q_2).
\end{align*}
To control each $\int \ell (\d Q_1 - \d Q_2)$, define the following approximation function of $\ell(\theta)$:
\begin{align*}
\theta\mapsto\hat{\ell}(\theta):=\sum_{x=0}^\infty b_xf(x|\theta), \text{ where }b_x\in\mathbb{R}\text{ and }\theta\in[0,\theta_*],
\end{align*}
Recall that $h_Q(x) = \int f(x|\theta)\d Q(\theta)$. Then direct calculation yields that
\begin{align*}
\int_0^{\theta_*}\ell(\theta)\d\big(Q(\theta)-\widehat{Q}(\theta)\big) &= \int_0^{\theta_*}\big(\ell(\theta)-\hat{\ell}(\theta)\big)\d\big(Q(\theta)-\widehat{Q}(\theta)\big) + \sum_{x=0}^\infty b_x \big(h_Q(x) - h_{\hat{Q}}(x)\big)\\
&\leq ~2\pnorm{\ell-\hat{\ell}}{\infty}+\Big|\sum_{x=0}^\infty b_x\big(h_Q(x)-h^{\rm obs}_Q(x)\big)\Big|+\Big|\sum_{x=0}^\infty b_x\big(h^{\rm obs}_Q(x)-h_{\widehat{Q}}(x)\big)\Big|,
\end{align*}
where $\pnorm{\ell-\hat{\ell}}{\infty}\equiv\sup_{\theta\in[0,\theta_*]}|\ell(\theta)-\hat{\ell}(\theta)|$ and $h^{\rm obs}_Q(x)\equiv n^{-1}\sum_{i=1}^n \bm{1}_{X_i = x}$. This implies
\begin{align}\label{Formula1}
W_1(Q,\widehat{Q})\leq \sup_{\ell\in \text{Lip}(1)}\Big\{2\pnorm{\ell-\hat{\ell}}{\infty}
+\Big|\sum_{x=0}^\infty b_x\big(h_{Q}(x)-h^{\rm obs}_Q(x)\big)\Big|
+\Big|\sum_{x=0}^\infty b_x\big(h^{\rm obs}_Q(x)-h_{\widehat{Q}}(x)\big)\Big|\Big\}.
\end{align}
By Lemmas \ref{Lemma:FirstPart} and \ref{Lemma:SecondPart}, for an arbitrary $\delta\in(0,1/2)$ and an arbitrary $\epsilon\in(0,1)$, there exist constants $n_1=n_1(\epsilon)$ and $C_1=C_1(\epsilon,\theta_*)$ such that the sum of the last two terms in (\ref{Formula1}) is upper bounded by 
\[
C_1\max_{x\geq 0}|b_x|/\sqrt{n^{1-\epsilon}\delta^{1+\epsilon}}
\]
for all $n\geq n_1$ with probability at least $1-2\delta$. The bound on $\max_{x\geq 0}|b_x|$ depends on the tail of $1/w(x)$.

(i) If $1/w(x)\leq C_2^{x}$ for some universal constant $C_2>1$ and all $x\geq 1$, it follows from Lemma~\ref{Prop:BoundOnf1} that 
any $1$-Lipschitz function $\ell(\theta)$ on $[0,\theta_*]$ can be approximated by 
$\hat{\ell}(\theta) = \sum_{x = 0}^{k} b_x f(x|\theta)$, such that $\max_{\theta\in[0,\theta_*]}|\ell(\theta)-\hat{\ell}(\theta)|\leq C_3/k$ and  
\begin{align*}
\max_{x\geq 0} |b_x|=\max_{x\in [0,k]} |b_x| \leq C_3^{k}/w(k) \leq (C_2C_3)^k
\end{align*}
for $k\geq 1$, where $C_3=C_3(\theta_*)>1$ is a constant. Hence it follows from (\ref{Formula1}) that
\[
W_1(Q,\widehat{Q})
\leq
2C_3/k+C_1(C_2C_3)^k/\sqrt{n^{1-\epsilon}\delta^{1+\epsilon}},
\]
for any $n\geq 1$ with probability at least $1-2\delta$.
Taking $k=k(n)$ such that $(C_2C_3)^k=n^{c}$ for some small positive constant $c$ specified later, it follows that
\begin{eqnarray*}
W_1(Q,\widehat{Q})
\leq
2C_3/k(n)+C_1n^c/\sqrt{n^{1-\epsilon}\delta^{1+\epsilon}}
=
2C_3/k(n)+C_1n^{c+\epsilon/2-1/2}/\sqrt{\delta^{1+\epsilon}}.
\end{eqnarray*}
Note that 
$
(C_2C_3)^{k(n)}=n^{c}
$
implies $k(n)= c\log n/\log (C_2C_3)$.
Letting $\epsilon=1/4$ and $c=1/8$, it follows that 
\begin{align*}
W_1(Q,\widehat{Q}) \leq &2C_3\log (C_2C_3)/(c\log n)+C_1n^{c+\epsilon/2-1/2}/\sqrt{\delta^{1+\epsilon}}\\
\leq& 16C_3\log (C_2C_3)/\log n+C_1n^{-1/4}/\delta^{5/8}.
\end{align*}
Therefore, for sufficiently large $n$ (depending on $\theta_*$), there exists a positive constant $C_4=C_4(\theta_*)$ such that $\E W_1(Q,\widehat{Q})\lesssim\log n$ by integrating the tail estimate.

\medskip

(ii)
If $1/w(x)\leq (C_5x)^{C_5x}$ for some universal constant $C_5$ and all $x\geq 1$, it follows from Lemma~\ref{Prop:BoundOnf1} that 
any $1$-Lipschitz function $\ell(\theta)$ on $[0,\theta_*]$ can be approximated by 
$\hat{\ell}(\theta) =\sum_{x = 0}^{k} b_xf(x|\theta)$ such that $\max_{\theta\in[0,\theta_*]}|\ell(\theta)-\hat{\ell}(\theta)|\leq C_3/k$, and
\begin{eqnarray*}
\max_{x} |b_x|
\leq C_3^{k}/w(k)
\leq (C_5(C_3)^{1/C_5}k)^{C_5k}
\leq (C_6k)^{C_6k}
\end{eqnarray*}
for $k\geq 1$, where $C_6=C_6(\theta_*)$ is a constant.
Hence it follows that
\[
W_1(Q,\widehat{Q})\leq 2C_3/k+(C_6k)^{C_6k}C_1/\sqrt{n^{1-\epsilon}\delta^{1+\epsilon}},
\]
for any $n\geq 1$ with probability at least $1-2\delta$. Taking $k=k(n)$ satisfying $(C_6k)^{C_6k}=n^{c}$ for a small positive constant $c$ specified later, it follows that
\begin{eqnarray*}
W_1(Q,\widehat{Q}) \leq 2C_3/k(n)+C_1n^{c+\epsilon/2-1/2}/\sqrt{\delta^{1+\epsilon}}.
\end{eqnarray*}
Since $\left(C_6k\right)^{C_6k}=n^{c}$ is equivalent to  $\log(C_6k)\exp(\log(C_6k))=c\log n$, it follows that $\log(C_6k(n))=W(c\log n)$ and hence $k(n)=\exp(W(c\log n))/C_6$, where $W(\cdot)$ is the Lambert W function. Using the expansion
\begin{align*}
W(x)=\log x-\log\log x+o(1), \text{ as }x\rightarrow\infty,
\end{align*}
and hence there exists a universal constant $C_7>0$ such that 
\[
\exp(W(x))\geq x/(2\log x) \text{ for }x\geq C_7.
\]
Therefore, for sufficiently large $n$, we have
\begin{eqnarray}
k(n)\geq\frac{c\log n}{2C_6\log (c\log n)}.
\end{eqnarray} 
As a result, 
\begin{eqnarray*}
W_1(Q,\widehat{Q})\leq\{4C_3C_6\log (c\log n)\}/(c\log n)+C_1n^{c+\epsilon/2-1/2}/\sqrt{\delta^{1+\epsilon}},
\end{eqnarray*}
with probability at least $1-2\delta$. Letting $c=1/8,\epsilon=1/4$, we have
\begin{align*}
W_1(Q,\widehat{Q}) \lesssim \log\log n/\log n+ n^{-1/4}\delta^{-5/8}.
\end{align*}
Therefore, for sufficiently large $n$ (depending on $\theta_*$), it follows that $\E W_1(Q,\widehat{Q}) \lesssim \log \log n/\log n$ by integrating the tail estimate. 
\end{proof}

\subsection{Proof of Theorem \ref{thm:main}}\label{subsec:proof_main}

\begin{proof}[Proof of Theorem \ref{thm:main}] 

The proof is separated to three steps. 

\par\noindent\textbf{Step 1.} 
In the first step, we prove that for any $\sigma>0$, integer $k>1$, and any $\ell \in\text{Lip}(1)$ on $[-\theta_*,\theta_*]$ with $\ell (0) = 0$, there exist a positive constant $C_4=C_4(\theta_*,\sigma)$ and a set of coefficients 
\[
\Big\{b_x\in\mathbb{R},x=0,\ldots,2k\Big\} 
\]
such that 
\begin{align*}
\sup_{\theta\in[0,\theta_*]}\Big|\ell _\sigma(\theta)-\ell _\sigma(0) -\sum_{0\leq x\leq 2k}b_xf(x|\theta)\Big| 
\leq C_3\Big\{\Big[\theta_*\sigma \sqrt{ek}\Big]^{-k}+\sum_{x\geq k+1}w(x)\theta_*^x\Big\},
\end{align*}
where we recall that $\ell _\sigma(\theta):= [\ell \ast\phi_\sigma](\theta)
$ and $\phi_\sigma$ is the probability density function of $\mathcal{N}_\sigma$.

For any $k=1,2,\ldots$, let $q_k(\theta):= \sum_{x=0}^k w(x)\theta^x$ be an approximation of the function $\theta\mapsto 1/g(\theta) = \sum_{x=0}^\infty w(x)\theta^x$ on $[0,\theta_*]$. 
Then one can readily verify that
\begin{align*}
R_k(\theta):=g(\theta)\cdot\Big\{\frac{1}{g(\theta)}-q_k(\theta)\Big\}
= g(\theta)\cdot\sum_{x\geq k+1} w(x)\theta^x
\leq g(0)\cdot\sum_{x\geq k+1}w(x)\theta_*^x
\end{align*}
whenever $\theta\in[0,\theta_*]$. 

Let $p_k(\theta)$ be the degree-$k$ polynomial achieving the approximation bound in Lemma \ref{lem:key}. We then have 
\begin{align}
\sup_{\theta\in[-\theta_*,{\theta_*}]}\Big|\ell _\sigma(\theta)-\ell _\sigma(0) - p_k(\theta)\Big|\leq C_5 e\sigma\cdot \Big[2\theta_*^{-1}\sigma \sqrt{ek}\Big]^{-k},\label{PolynomialAchievingell_sigma}
\end{align}
where $C_5>0$ is a universal constant.

Let's construct $\{b_x\in\mathbb{R},x\in[2k]\}$ to be coefficients such that 
\[
p_k(\theta)q_k(\theta)=\sum_{x=0}^{2k}b_xw(x)\theta^x.
\]
Then have $g(\theta)p_k(\theta)q_k(\theta) = \sum_{x=0}^{2k} b_xf(x|\theta)$, 
and the proof in this step is complete by noting that 
\begin{align}
&\quad\sup_{\theta\in[0,{\theta_*}]}\Big|\ell _\sigma(\theta)-\ell _\sigma(0) - p_k(\theta)q_k(\theta)g(\theta)\Big|\notag\\
&= \sup_{\theta\in[0,{\theta_*}]}\Big|\ell _\sigma(\theta) -\ell _\sigma(0)- p_k(\theta)\big[1 - R_k(\theta)\big]\Big|\notag\\
&\leq 2\sup_{\theta\in[0,{\theta_*}]}\Big|\ell _\sigma(\theta)-\ell _\sigma(0) - p_k(\theta)\Big| + \sup_{\theta\in[0,{\theta_*}]}\Big|\ell _\sigma(\theta)-\ell _\sigma(0)\Big|\cdot\sup_{\theta\in[0,{\theta_*}]}\Big|R_k(\theta)\Big|\notag\\
&\stackrel{(*)}{\leq} C_5 2e\sigma\cdot \Big[2\theta_*^{-1}\sigma \sqrt{ek}\Big]^{-k}+2(\theta_*+ \sigma)g(0)\cdot\sum_{x\geq k+1}w(x)\theta_*^x\label{eq:Upper1InProofOfthm:main}
\end{align}
and 
\[
\sup_{\theta\in[0,{\theta_*}]}g(\theta)p_k(0)q_k(0)\leq g(0)p_k(0)q_k(0).
\]
Here in $(*)$ we use the fact that, as $\ell (0) = 0$ and $\ell \in\text{Lip}(1)$,
\begin{align*}
\sup_{\theta\in[-\theta_*,\theta_*]}\Big|\ell _\sigma(\theta)\Big| = \sup_{\theta\in[-\theta_*,\theta_*]}\Big|\int \ell (\theta-\theta_1)\phi_\sigma(\theta_1){\sf d} \theta_1\Big| \leq \int ({\theta_*}+|\theta_1|)\phi_\sigma(\theta_1){\sf d} \theta_1\leq \theta_*+ \sigma.
\end{align*}
\medskip

\par\noindent\textbf{Step 2.} In this step, we upper bound $\max_{x\in[2k]}|b_x|$. Let 
\[
\tilde{r}(\theta):=p_k(\theta_*\theta)q_k(\theta_*\theta):=\sum_{x=1}^{2k}\tilde{b}_xw(x)\theta^x 
\]
be a rescaled version of $p_k(\theta)q_k(\theta)$, so that $\tilde{b}_x = \theta_*^xb_x$. Then by Lemma~\ref{fact:coeff-bound}, it holds that for each $1\leq x\leq 2k$,
\begin{align*}
\Big|\tilde{b}_x\Big|w(x) \leq \frac{(2k)^x}{x!}\sup_{|\theta|\leq 1}\Big|\tilde{r}(\theta)\Big| \leq \frac{(2k)^x}{x!}\sup_{|\theta|\leq \theta_*}p_k(\theta)\cdot \sup_{|\theta|\leq \theta_*}q_k(\theta).
\end{align*}
Since 
\[
\sup_{|\theta|\leq {\theta_*}}q_k(\theta)\leq 1/g(\theta_*) 
\]
and by \eqref{PolynomialAchievingell_sigma}, 
\[
\sup_{|\theta|\leq \theta_*}p_k(\theta) \leq C
\]
for some positive constant $C$ only depending on $\theta_*$ and $\sigma$, it follows that 
\begin{align*}
\max_{1\leq x\leq 2k}\Big|b_x\Big|
\leq C_6\max_{1\leq x\leq 2k}\frac{(2k)^x}{w(x)\theta_*^xx!}
\leq C_6\max_{1\leq x\leq 2k}\frac{1}{w(x)}\cdot \max_{1\leq x\leq 2k}\frac{1}{\theta_*^x}\cdot \max_{1\leq x\leq 2k}\frac{(2k)^x}{x!}
\end{align*}
where $C_6=C_6(\theta_*,\sigma)>0$. Combining the above inequality with 
\[
\max_{1\leq x\leq 2k}1/\theta_*^x\leq \Big(\max\{1,1/\theta_*\}\Big)^{2k}
\]
and 
\[
\max_{1\leq x\leq 2k}(2k)^x/x! \leq e^{2k}, 
\]
it follows that 
\[
\max_{1\leq x\leq 2k}\Big|b_x\Big|\leq C_6\cdot\Big(e\cdot\max\{1,1/\theta_*\}\Big)^{2k}\cdot \max_{1\leq x\leq 2k}\frac{1}{w(x)}.
\]
\medskip

\par\noindent
\textbf{Step 3.} In this step we prove the claim of the theorem. Recall that 
\[
W_1^{\sigma}(\hat{Q},Q)=\sup_{\ell}\int \ell {\sf d}[\hat{Q}\ast\mathcal{N}_\sigma] - \ell{\sf d}[Q\ast\mathcal{N}_\sigma],
\]
where $\ell\in\text{Lip}(1)$ with $\ell (0)=0$. It further holds that
\begin{align*}
 W_1^{\sigma}(\hat{Q},Q)& = \sup_{\ell \in\text{Lip}(1):\ell (0)=0}\int \ell {\sf d}[\hat{Q}\ast\mathcal{N}_\sigma] - \ell{\sf d}[Q\ast\mathcal{N}_\sigma]\\ 
& = \sup_{\ell \in\text{Lip}(1):\ell (0)=0}\int (\ell _\sigma(\theta)-\ell _\sigma(0)) [{\sf d}\hat{Q} - {\sf d}Q]\\
& = \sup_{\ell \in\text{Lip}(1):\ell (0)=0}\int \Big\{\ell _\sigma(\theta)-\ell _\sigma(0)-\sum_{0\leq x\leq 2k}b_xf(x|\theta)\Big\}[{\sf d}\hat{Q}- {\sf d}Q] + \\
&\quad\quad\sup_{\ell \in\text{Lip}(1):\ell (0)=0}\int \sum_{0\leq x\leq 2k}b_xf(x|\theta)[{\sf d}\hat{Q}-{\sf d}Q]\\
& := (I) + (II).
\end{align*}
By Step 2, we have
\begin{align}\label{ineq:bound_I}
(I)
\leq 2C_4\Big\{\Big[2\theta_*^{-1}\sigma \sqrt{ek}\Big]^{-k}+\sum_{x\geq k+1}w(x)\theta_*^x\Big\}.
\end{align}
Next we bound $(II)$. Recall that 
\[
h^{\rm obs}(x):=\sum_{i=1}^n\ind(X_i=x)/n .
\]
We have
\begin{align*}
\int \sum_{0\leq x\leq 2k}b_xf(x|\theta)[{\sf d}\hat{Q}(\theta) - {\sf d}Q(\theta)]
&\leq\Big|\sum_{0\leq x\leq 2k} b_x [h_{\hat{Q}}(x) - h^{\rm obs}(x)]\Big| + \Big|\sum_{0\leq x\leq 2k}b_x[h^{\rm obs}(x) - h_Q(x)]\Big|.
\end{align*}
It follows from Lemma~\ref{Lemma:FirstPart} that for any $\delta > 0$, it holds with probability $1-\delta$ that
\begin{align*}
\Big|\sum_{0\leq x\leq 2k}b_x\Big[h^{\rm obs}(x) - h_Q(x)\Big]\Big| \leq \max_{0\leq x\leq 2k}|b_x|\sqrt{\frac{\log(2/\delta)}{2n}}. 
\end{align*}
Moreover, it follows from Lemma~\ref{Lemma:SecondPart} that for an arbitrary $\delta\in(0,1)$ and an arbitrary $\epsilon\in(0,1)$, there exists a constant $C_7=C_7(\epsilon,\theta_*)>0$ such that
\begin{eqnarray*}
\left|\sum_{x=0}^{2k} b_x\left\{h^{\rm obs}_Q(x)-h_{\widehat{Q}}(x)\right\}\right|
\leq C_7\max_{0\leq x\leq 2k}|b_x|\sqrt{\frac{1}{n^{1-\epsilon}\delta^{1+\epsilon}}}
\end{eqnarray*}
holds with probability at least $1-\delta$.

Consequently, we have 
\[
(II)\leq C_8\max_{0\leq x\leq 2k}|b_x|\Big/\sqrt{n^{1-\epsilon}\delta^{1+\epsilon}}
\]
with probability at least $1-\delta$ for some constant $C_8=C_8(\epsilon,\theta_*)>0$.
Note that $\max_{0\leq x\leq 2k}|b_x|$ have been upper bounded in Step 2.

Putting together the estimates for $(I)$ and $(II)$, we have that with probability at least $1-\delta$, $W_1^\sigma(Q,\hat Q)$ is upper bounded by 
\begin{align}
\Big[2\theta_*^{-1}\sigma \sqrt{ek}\Big]^{-k}+\sum_{x\geq k+1}w(x)\theta_*^x+\Big(e\cdot\max\{1,1/\theta_*\}\Big)^{2k}\cdot \max_{1\leq x\leq 2k}\frac{1}{w(x)}\Big/\sqrt{n^{1-\epsilon}\delta^{1+\epsilon}}.
\label{eq:TradeOff}
\end{align}
up to a constant depending on $\sigma,\theta_*$ and $\epsilon$.
\medskip

\vspace{0.2cm}

(i) If $c_1c_2^x\leq 1/w(x)\leq C_1C_2^x$, \eqref{eq:TradeOff} becomes 
\[
[2\theta_*^{-1}\sigma \sqrt{ek}]^{-k}+\sum_{x\geq k+1}w(x)\theta_*^x+C_9^{2k}/\sqrt{n^{1-\epsilon}\delta^{1+\epsilon}},
\]
where $C_9:=e\cdot\max\{1,1/\theta_*\}\cdot\max\{1,C_2\}$ is a positive constant.
For the second term, it follows from \citet[Corollary 1.1.10]{krantz2002primer} that for any $R\in(\theta_*,\theta_r)$ there exists some constant  $C_{10}=C_{10}(R)>0$ such that $w(x)\leq C_{10}/R^x$ for all $x=0,1,2\ldots$, and hence 
\[
\sum_{x\geq k+1}w(x)\theta_*^x
\leq C_{10}\sum_{x\geq k+1}(\theta_*/R)^x
\leq C_{10}\cdot(\theta_*/[R-\theta_*])\cdot[\theta_*/R]^k
\text{ for any }k=1,2,\ldots.
\]
Therefore, the second term dominates the first term in \eqref{eq:TradeOff}, and \eqref{eq:TradeOff} becomes
\[
[\theta_*/R]^k+C_9^{2k}/\sqrt{n^{1-\epsilon}\delta^{1+\epsilon}}.
\]
The proof is then complete by letting $C_9^{2k}= n^{\alpha}$ for some $\alpha \in(0,1/2-\epsilon/2)$. The final bound is then $n^{-\frac{(1-\epsilon)\log(R/\theta_*)}{2\log(R/\theta_*)+4\log C_9}}$ for any $\epsilon\in(0,1)$.
\medskip

\vspace{0.2cm}

(ii) If $c_1c_2^xx^{c_3x}\leq 1/w(x)\leq C_1C_2^xx^{C_3x}$, \eqref{eq:TradeOff} becomes
\[
[2\theta_*^{-1}\sigma \sqrt{ek}]^{-k}+(C_{11}k)^{-c_3k}+(C_{12}k)^{2C_3k}/\sqrt{n^{1-\epsilon}\delta^{1+\epsilon}}
\]
for some positive constants $C_{11}$ and $C_{12}$
and the proof is then complete by letting $(C_{12}k)^{2C_3k} = n^{\alpha}$ for some $\alpha \in(0,1/2-\epsilon/2)$. The final bound is then $n^{-\frac{(1-\epsilon)/2}{1+\max\{4C_3,2C_3/c_3\}}}$.
\end{proof}

\section{Auxiliary results}\label{sec:auxiliary}

\subsection{Auxiliary lemmas}

\begin{lemma}[Theorem 6.2 in Chapter 7, \cite{devore1976degree}]\label{lem:devore}
For any integer $r\geq 1$, let 
\begin{align*}
W^r_\infty([-1,1]):= \Big\{\psi:[-1,1]\rightarrow\R: \psi^{(r-1)} &\text{ is absolutely continuous and} \\
&\text{the supremum of } \psi^{(r)} \text{ on }[-1,1] \text{ is finite}\Big\}
\end{align*}
be the Sobolev space on $[-1,1]$.
For functions $f\in W^r_\infty([-1,1])$ and any integer $k>r$, there exists a polynomial $p_k$ of degree at most $k$ such that 
\[
\sup_{\theta\in[-1,1]}\Big|f(\theta)-p_k(\theta)\Big| \leq Ck^{-r}\omega\big(f^{(r)},k^{-1}\big),
\]
where $C>0$ is a universal constant and 
\[
\omega(f^{(r)},k^{-1}):= \sup_{\theta_1,\theta_2:|\theta_1-\theta_2|\leq k^{-1}}\Big|f^{(r)}(\theta_1)-f^{(r)}(\theta_2)\Big|.
\]
\end{lemma}

Recall that for a sample $\{X_i, i\in[n]\}$ and $x\in\mathbb{N}$, $h^{\rm obs}_Q(x) = n^{-1}\sum_{i=1}^n \bm{1}_{X_i = x}$. The following lemmas provides the concentration of $h^{\rm obs}_Q$ around $h_Q$.
\begin{lemma}[Lemma A.1 in \cite{miao2021fisher}]
Let $\{X_i,i\in[n]\}$ be an i.i.d. sample generated from the probability mass function $h_Q$ in (\ref{eq:model}). Then for any $\delta\in(0,1)$ the following inequality holds with probability at least $1-\delta$,
\[
\left|\sum_{x=0}^\infty b_x\left(h^{\rm obs}_Q(x)-h_{Q}(x)\right)\right| \leq  \max_{x\geq0}|b_x|\sqrt{\frac{\log(2/\delta)}{2n}},
\]
where $b_x\in\mathbb{R}$ for all $x\in \mathbb{N}$.
\label{Lemma:FirstPart}
\end{lemma}

\begin{lemma}[A generalized version of Lemma A.2 in \cite{miao2021fisher}] Let $\{X_i,i\in[n]\}$ be an i.i.d. sample generated from the mixture distribution $h_Q$ in (\ref{eq:model}). Then for an arbitrary $\delta\in(0,1)$ and an arbitrary $\epsilon\in(0,1)$, there exists a constant $C=C(\epsilon,\theta_*)>0$ such that for any $n\geq 1$,
\begin{eqnarray*}
\left|\sum_{x=0}^\infty b_x\left(h^{\rm obs}_Q(x)-h_{\widehat{Q}}(x)\right)\right|
\leq C_1\max_{x\geq 0}|b_x|\sqrt{\frac{1}{n^{1-\epsilon}\delta^{1+\epsilon}}}
\end{eqnarray*}
holds with probability at least $1-\delta$. Here $b_x\in\mathbb{R}$ for all $x\in \mathbb{N}$.
\label{Lemma:SecondPart}
\end{lemma}

\begin{lemma}[A generalized version of Proposition A.2. in \cite{miao2021fisher}]
\label{Prop:BoundOnf1}
For any $1$-Lipschitz function $\theta\mapsto\ell(\theta)$ on $[0,\theta_*]$ with $\ell(0)=0$, there exists some $\hat{\ell}(\theta) = \sum_{x = 0}^{k} b_xf(x|\theta)$ such that $\max_{\theta\in[0,\theta_*]}|\ell(\theta) - \hat{\theta}|\leq C/k$, and 
\[
\max_{x\in[0,k]} |b_x| \leq C^{k}\cdot\max_{1\leq x\leq k}1/w(x) \text{ for } k\geq 1,
\]
where $C=C(\theta_*)$ is a positive constant. It can be further proved that there exists some universal constant $C'>0$ such that 
\[
C^{x}/w(x)\geq e^x/C' 
\]
for all nonnegative integer $x$.
\end{lemma}

\begin{lemma}[Chapter 2.6 Equation 9 in \cite{timan2014theory}]
\label{fact:coeff-bound}
Suppose $k$ is a non-negative integer and $\theta\mapsto p_k(\theta)\equiv \sum_{x=0}^k c_x\theta^x$. Then it follows that coefficients $\{c_x\}_{x=0}^k$ satisfy
\[
|c_x|\le \frac{k^x}{x!}\max_{|\theta|\le 1}|p_k(\theta)|.
\]
\end{lemma}
\begin{lemma}[Jackson’s theorem, Lemma 10 of \cite{han2020optimality} or see~\cite{devore1976degree}]
\label{lem:jackson} 
Let $k>0$ be any integer, and $[a, b] \subseteq \mathbb{R}$ be any
bounded interval. For any $1$-Lipschitz function $\ell(\cdot)$ on $[a, b]$, there exists a universal constant $C$ independent of $k, \ell$ such that there exists a polynomial $p_k(\cdot)$ of degree at most $k$ such that
\begin{align}
|\ell(\theta)-p_k(\theta)|\leq C\sqrt{(b-a)(\theta-a)}/k, \; \forall \theta\in [a,b].
\end{align}
In particular, the following norm bound holds:
\begin{align}
\sup_{\theta\in [a,b]}|\ell(\theta)-p_k(\theta)|\leq C(b-a)/k.
\end{align}
\end{lemma}

\subsection{Proofs of Remarks}

\begin{proof}[Proof of Remark \ref{remark:UB}:]
(1) If $1/w(x)\leq\exp_L(C_9x)$ for some universal constant $C_9$ and all $x\geq 1$, it follows from Lemma~\ref{Prop:BoundOnf1} that 
any $1$-Lipschitz function $\ell(\theta)$ on $[0,\theta_*]$ can be approximated by 
$\hat{\ell}(\theta) = g(\theta)\sum_{x = 0}^{k} b_xw(x)\theta^x$ with an uniform approximation error of
 $C_3/k$ with 
\begin{eqnarray*}
\max_{x} |b_x|
\leq C_3^{k}/w(k)
\leq C_3^{k}\exp_L(C_9k)
\leq \exp_L(C_{10}k)
\end{eqnarray*}
for $k\geq 1$, where $C_{10}=C_{10}(\theta_*)$ is a constant.
Hence it follows from the first steps in the proof of Theorem~\ref{thm:ub-W1} that
\[
W_1(Q,\widehat{Q})
\leq
2C_3/k+C_1\exp_L(C_{10}k)/\sqrt{n^{1-\epsilon}\delta^{1+\epsilon}},
\]
for $n\geq n_1$ with probability at least $1-2\delta$.
Analogously, by letting $k=k(n)$ such that $\exp_L(C_{10}k)=n^c$ for a small $c$, we then have 
$
E\{W_1(Q,\widehat{Q})\}\leq C_{11}/\log_L (n),
$
where $C_{11}=C_{11}(\theta_*)$ is a constant.
\medskip

\par\noindent
(2) If $1/w(x)\leq (C_{12}x)^{\cdots^{(C_{12}x)}}$ ($L\in\mathbb{N}^+$ times power) for some universal constant $C_{12}$ and all $x\geq 1$, it follows from Lemma~\ref{Prop:BoundOnf1} that 
any $1$-Lipschitz function $\ell(\theta)$ on $[0,\theta_*]$ can be approximated by 
$\hat{\ell}(\theta) = g(\theta)\sum_{x = 0}^{k} b_xw(x)\theta^x$ with an uniform approximation error of
 $C_3/k$ with 
\begin{eqnarray*}
\max_{x} |b_x|
\leq C_3^{k}/w(k)
\leq C_3^{k} (C_{12}k)^{\cdots^{(C_{12}k)}}
\leq (C_{13}k)^{\cdots^{(C_{13}k)}}
\end{eqnarray*}
for $k\geq1$, where $C_{13}=C_{13}(\theta_*)$ is a constant.
Hence it follows from the first steps in the proof of Theorem~\ref{thm:ub-W1} that
\[
W_1(Q,\widehat{Q})
\leq
2C_3/k+C_1(C_{13}k)^{\cdots^{(C_{13}k)}}/\sqrt{n^{1-\epsilon}\delta^{1+\epsilon}},
\]
for $N\geq N_1$ with probability at least $1-2\delta$.
By letting $k=k(n)$ such that $(C_{13}k)^{\cdots^{(C_{13}k)}}=n^c$ for a small $c$, we have 
$
E\{W_1(Q,\widehat{Q})\}\leq C_{14}\log_{L} (n)/\log_{L-1} (n),
$
where $C_{14}=C_{14}(\theta_*)$ is a constant.
\end{proof}

\subsection{Proofs of Lemmas}

\begin{proof}[Proof of Lemma \ref{lem:simple}]
Recall the duality definition of $W_1(\mu_1,\mu_2)$ as
\[
W_1(\mu_1,\mu_2):=\inf \E\|X-Y\|,
\]
with the infimum taken over all couplings of $(X,Y)$ such that $X\sim \mu_1$ and $Y\sim \mu_2$. We then consider any such $(X,Y)$ and assume $Z$ to be independent of $(X,Y)$ and follows the distribution of $\nu$. Then it is immediate that
\[
W_1(\mu_1\ast \nu,\mu_2\ast \nu)\leq \E\|(X+Z)-(Y+Z)\|=\E\|X-Y\|,
\]
and accordingly (by taking infimum over all such $(X,Y)$) 
\[
W_1(\mu_1\ast \nu,\mu_2\ast \nu)\leq  W_1(\mu_1,\mu_2).
\]
This completes the proof.
\end{proof}

\begin{proof}[Proof of Lemma \ref{lem:key}]
By rescaling, we assume that $a = -1$ and $b=1$. 
For any integer $r\geq 1$, let
\begin{align*}
W^r_\infty([a,b]):= \Big\{\psi:[a,b]\rightarrow\R: \psi^{(r-1)} &\text{ is absolutely continuous  and}\\ 
&\text{the essential supremum of } \psi^{(r)} \text{ on }[a,b] \text{ is finite}\Big\}
\end{align*}
be the Sobolev space on $[a,b]$. Then it is readily verifiable that for any $\ell\in\textrm{Lip}(1)$ and $\sigma^2>0$,
\[
\ell _\sigma(\theta)-\ell _\sigma(0)=(\ell \ast\phi_\sigma)(\theta)-(\ell\ast \phi_\sigma)(0), 
\]
when restricted on $[a,b]$, belongs to $W^r_\infty([a,b])$. 
Hence by Lemma \ref{lem:devore}, we have that for any integer $k> r$, there exists some polynomial $p_k$ of degree $k$ such that
\begin{align*}
\sup_{\theta\in[a,b]}\Big|\ell _\sigma(\theta)-\ell _\sigma(0) - p_k(\theta)\Big|\leq C_1k^{-r}\omega\big(\ell _\sigma^{(r)},k^{-1}\big),
\end{align*}
In the above inequality, $C_1=C_1(a,b) > 0$ is a constant and 
\[
\omega(\psi,t):= \sup_{\theta_1,\theta_2:|\theta_1-\theta_2|\leq t}|\psi(\theta_1)-\psi(\theta_2)| 
\]
is the modulus of continuity of function $\psi$ at radius $t$. To bound the righthand side of the above display, note that, with $H_n(\cdot)$ denoting the $n$-th Hermite polynomial, we have
\begin{align*}
\ell_\sigma^{(r)}(\theta) &= \int \ell ({\theta_1})\phi_\sigma^{(r)}(\theta-{\theta_1}){\sf d} {\theta_1}
= \sigma^{-r}(-1)^r \int \ell (\theta-{\theta_1})\phi_\sigma({\theta_1})H_r\big({\theta_1}/\sigma\big){\sf d} {\theta_1}.
\end{align*}
Hence for any $\theta_1,\theta_2$ such that $|\theta_1-\theta_2|\leq k^{-1}$, we have
\begin{align*}
\big|\ell_\sigma^{(r)}(\theta_1)-\ell_\sigma^{(r)}(\theta_2)\big|
&\leq \sigma^{-r}\int |\ell (\theta_1-\theta)-\ell (\theta_2-\theta)| \phi_\sigma(\theta)|H_r(\theta/\sigma)|{\sf d} \theta\\
&\leq \sigma^{-r}k^{-1}\int \phi_\sigma(\theta)|H_r(\theta/\sigma)|{\sf d} \theta 
= \sigma^{-r}k^{-1}\int \phi_1(\theta)|H_r(\theta)\big|{\sf d} \theta\\
& \leq \sigma^{-r}k^{-1}[\int \phi_1(\theta)H_r^2(\theta){\sf d} \theta]^{1/2} = \sigma^{-r}k^{-1}\sqrt{r!}.
\end{align*}
It further follows from the Sterling formula $\sqrt{r!}\leq\sqrt{er^{r+1/2}e^{-r}}$ that 
\[
\big|\ell_\sigma^{(r)}(\theta_1)-\ell_\sigma^{(r)}(\theta_2)\big|
\leq \sigma^{-r}k^{-1}\sqrt{er^{r+1/2}e^{-r}}.
\]
Using $r < k$, we hence obtain
\[
\sup_{\theta\in[a,b]}|\ell _\sigma(\theta)-\ell _\sigma(0) - p_k(\theta)|
\leq C_1\sqrt{e}(\sqrt{e}\sigma k/\sqrt{r})^{-r}r^{1/4}k^{-1}
\leq  C_1\sqrt{e}(\sqrt{e}\sigma \sqrt{k})^{-r}k^{-3/4}.
\]
By rescaling, we then have for any $a\leq0,b\geq0$ it follows that 
\begin{align*}
\sup_{\theta\in[a,b]}|\ell _\sigma(\theta)-\ell _\sigma(0) - p_k(\theta)|
&\leq C_1([b-a]/2)^{r+1}\sqrt{e}(\sqrt{e}\sigma \sqrt{k})^{-r}k^{-3/4}\\
&\leq \frac{C_1(b-a)\sqrt{e}}{2}\cdot \Big[\frac{2\sqrt{e}\sigma \sqrt{k}}{b-a}\Big]^{-r}k^{-3/4}.
\end{align*}
Now taking $r=k-1$, we have
\[
\sup_{\theta\in[a,b]}|\ell _\sigma(\theta)-\ell _\sigma(0) - p_k(\theta)|
\leq
C_1e\sigma\cdot  \Big[\frac{2\sqrt{e}\sigma \sqrt{k}}{b-a}\Big]^{-k}k^{-1/4}
\]
and accordingly complete the proof.
\end{proof}

\begin{proof}[Proof of Lemma~\ref{Lemma:SecondPart}.]
Whenever there is no ambiguity, let $h^{\rm obs}_Q$, $h_{\widehat{Q}}$, and $h_Q$ also represent distributions with respect to corresponding probability mass functions $x\mapsto h^{\rm obs}_Q(x)$, $x\mapsto h_{\widehat{Q}}(x)$, and $x\mapsto h_Q(x)$.

This proof consists of two steps. In the first step, we prove that 
\[
\left|\sum_{x=0}^\infty b_x\left(h^{\rm obs}_Q(x)-h_{\hat{Q}}(x)\right)\right|
\]
can be upper bounded by $\text{KL}(h^{\rm obs}_Q,h_Q)$, where KL is the Kullback–Leibler divergence.
 In the second step, we upper bound $\text{KL}(h^{\rm obs}_Q,h_Q)$ by truncation arguments.
\medskip

\par\noindent
\textbf{Step 1.}
It follows from the triangle inequality that 
\begin{eqnarray*}
\left|\sum_{x=0}^\infty b_x\left(h^{\rm obs}_Q(x)-h_{\widehat{Q}}(x)\right)\right|
\leq\max_{x\geq 0}|b_x|\sum_{x=0}^\infty \left|h^{\rm obs}_Q(x)-h_{\widehat{Q}}(x)\right|
=\max_{x\geq 0}|b_x|\cdot\Big\|h^{\rm obs}_Q-h_{\widehat{Q}}\Big\|_1,
\end{eqnarray*}
where $\|h^{\rm obs}_Q-h_{\widehat{Q}}\|_1$ represents the total variation distance between distributions $h^{\rm obs}_Q$ and $h_{\widehat{Q}}$. It further follows from Pinsker’s inequality that 
\begin{eqnarray*}
\Big\|h^{\rm obs}_Q-h_{\widehat{Q}}\Big\|_1\leq\sqrt{\frac{1}{2} \cdot\text{KL}(h^{\rm obs}_Q,h_{\widehat{Q}})},
\end{eqnarray*}
and hence 
\begin{eqnarray*}
\left|\sum_{x=0}^\infty b_x\left(h^{\rm obs}_Q(x)-h_{\widehat{Q}}(x)\right)\right|
\leq\max_{x\geq 0}|b_x|\sqrt{\frac{1}{2} \cdot\text{KL}(h^{\rm obs}_Q,h_{\widehat{Q}})}
\leq\max_{x\geq 0}|b_x|\sqrt{\frac{1}{2} \cdot\text{KL}(h^{\rm obs}_Q,h_Q)},
\end{eqnarray*}
by noting that maximum likelihood estimators maximize likelihood functions.

\vspace{0.2cm}

\par\noindent
\textbf{Step 2.}
Suppose $C_1=C_1(\theta_*)$ is the smallest positive integer larger than $\theta_*g(0)(1/g)^\prime(\theta_*)$. Define 
\[
T_i:= X_i\ind(X_i\leq C_1-1)+C_1\ind(X_i\geq C_1) \text{ for all } i\in[N].
\]
Let $t_Q$ be the probability mass function of $T_1$ and let $t_Q^{\rm obs}$ be the sample version of $t_Q$, i.e. 
\[
x\mapsto t_Q(x):=P(T_1=x) \text{ and } x\mapsto t_Q^{\rm obs}(x):= \frac{1}{n}\sum_{i=1}^n\ind(T_i=x), \text{ for }x\in \{0,\ldots,C_1\}.
\]
Note that 
\[
t_Q(x)=h_Q(x) ~~\text{ and }~~t_Q^{\rm obs}(x)=h_Q^{\rm obs}(x) \text{ for }x=0,\ldots,C_1-1 
\]
and 
\[
t_Q(C_1)=\sum_{x\geq C_1}h_Q(x), ~~~t_Q^{\rm obs}(C_1)=\sum_{x\geq C_1}h_Q^{\rm obs}(x).
\]
Hence it follows that 
\begin{eqnarray*}
\text{KL}(h^{\rm obs}_Q,h_Q)
&=&\sum_{x=0}^{C_1-1}t^{\rm obs}_Q(x)\log\frac{t^{\rm obs}_Q(x)}{t_Q(x)}+\sum_{x\geq C_1}h^{\rm obs}_Q(x)\log\frac{h^{\rm obs}_Q(x)}{h_Q(x)}\\
&=&\text{KL}(t_Q^{\rm obs},t_Q)-t^{\rm obs}_Q(C_1)\log\frac{t^{\rm obs}_Q(C_1)}{t_Q(C_1)}+\sum_{x\geq C_1}h^{\rm obs}_Q(x)\log\frac{h^{\rm obs}_Q(x)}{h_Q(x)},
\end{eqnarray*}
where $t_Q^{\rm obs}$ and $t_Q$ are viewed as distributions with respect to corresponding probability mass functions of $x\mapsto t_Q(x)$ and $x\mapsto t^{\rm obs}_Q(x)$.

If $t^{\rm obs}_Q(C_1)=0$, then 
\[
t^{\rm obs}_Q(C_1)\log\frac{t^{\rm obs}_Q(C_1)}{t_Q(C_1)}=0.
\]
Otherwise it follows from the inequality 
\[
\log(1+x)\leq x\text{ for }x>0
\] 
that 
\begin{eqnarray*}
-t^{\rm obs}_Q(C_1)\log\frac{t^{\rm obs}_Q(C_1)}{t_Q(C_1)}
\leq\sum_{x\geq C_1}\Big\{h_Q(x)-h_Q^{\rm obs}(x)\Big\}.
\end{eqnarray*}
Analogously, we have
\begin{eqnarray*}
\sum_{x\geq C_1}h^{\rm obs}_Q(x)\log\frac{h^{\rm obs}_Q(x)}{h_Q(x)}
\leq\sum_{x\geq C_1}\frac{(h^{\rm obs}_Q(x)-h_Q(x))^2}{h_Q(x)}+\sum_{x\geq C_1}\Big\{h^{\rm obs}_Q(x)-h_Q(x)\Big\}
\end{eqnarray*}
and hence
\begin{eqnarray*}
-t^{\rm obs}_Q(C_1)\log\frac{t^{\rm obs}_Q(C_1)}{t_Q(C_1)}+\sum_{x\geq C_1}h^{\rm obs}_Q(x)\log\frac{h^{\rm obs}_Q(x)}{h_Q(x)}
\leq \sum_{x\geq C_1}\frac{(h^{\rm obs}_Q(x)-h_Q(x))^2}{h_Q(x)}.
\end{eqnarray*}

\vspace{0.2cm}

\noindent\textbf{Step 2(a).} We first upper bound $\sum_{x\geq C_1}(h^{\rm obs}_Q(x)-h_Q(x))^2/h_Q(x)$. Fix an arbitrary $\epsilon\in(0,1)$ and choose a $\gamma>0$ in $(1-\epsilon,1)$.
Define $A:= \alpha^{(1-\gamma)/3}$, where $\alpha:=(\theta_*+\theta_r)/(2\theta_*)>1$.
Note that $\alpha\theta_*<\theta_r$ and we have $1/g(\theta)=\sum_{x=0}^\infty w(x)\theta^x<\infty$ for all $\theta\in[0,\alpha\theta_*]$.
It then follows from H{\"o}lder's inequality that 
\begin{align*}
&n^{1-\epsilon}\sum_{x\geq C_1}\frac{(h^{\rm obs}_Q(x)-h_Q(x))^2}{h_Q(x)}
=n^{1-\epsilon}\sum_{x\geq C_1}\frac{(h^{\rm obs}_Q(x)-h_Q(x))^2}{h_Q(x)}A^{-x}A^{x}\\
&\leq n^{1-\epsilon}\left(\sum_{x\geq C_1}\frac{(h^{\rm obs}_Q(x)-h_Q(x))^2}{h_Q(x)}A^{-x/\gamma}\right)^\gamma
\left(\sum_{x\geq C_1}\frac{(h^{\rm obs}_Q(x)-h_Q(x))^2}{h_Q(x)}A^{x/(1-\gamma)}\right)^{1-\gamma}.
\end{align*}
It further follows from $A>1$ that 
\[
n\cdot \E\Big\{\sum_{x\geq C_1}\frac{(h^{\rm obs}_Q(x)-h_Q(x))^2}{h_Q(x)}A^{-x/\gamma}\Big\}
=\sum_{x\geq C_1}(1-h_Q(x))A^{-x/\gamma}
\leq
\sum_{x\geq C_1}A^{-x/\gamma}
=\frac{A^{-C_1/\gamma}}{1-A^{-1/\gamma}}
<\infty
\]
and hence for an arbitrary $\delta\in(0,1)$, we have
\[
n\sum_{x\geq C_1}\frac{(h^{\rm obs}_Q(x)-h_Q(x))^2}{h_Q(x)}A^{-x/\gamma}\leq \frac{A^{-C_1/\gamma}}{1-A^{-1/\gamma}}\frac{1}{\delta}
\]
with probability at least $1-\delta$.
Therefore, with probability at least $1-\delta$, we have
\[
\left(\sum_{x\geq C_1}\frac{(h^{\rm obs}_Q(x)-h_Q(x))^2}{h_Q(x)}A^{-x/\gamma}\right)^{\gamma}\leq \left(\frac{A^{-C_1/\gamma}}{1-A^{-1/\gamma}}\frac{1}{N\delta}\right)^\gamma
\leq\frac{1}{\left(a^{\frac{1-\gamma}{3\gamma}}-1\right)^\gamma}\frac{1}{(N\delta)^\gamma},
\]
where the last inequality follows from $C_1\geq 1$.
On the other hand,
\begin{eqnarray*}
\sum_{x\geq C_1}\frac{(h^{\rm obs}_Q(x)-h_Q(x))^2}{h_Q(x)}A^{x/(1-\gamma)}
\leq\sum_{x\geq C_1}\frac{(h^{\rm obs}_Q(x))^2}{h_Q(x)}\alpha^{x/3}+\sum_{x\geq C_1}h_Q(x)\alpha^{x/3}.
\end{eqnarray*}
We first show that the second term on the right hand side is bounded, which is true if
\[
\sum_{x\geq C_1}h_Q(x)\alpha^{x}\leq g(\theta_*)\Big/g(\alpha\theta_*).
\]
Since $\sum_{x=0}^\infty g(\theta)w(x)\theta^x=1$ and $1/g(\theta)=\sum_{x=0}^\infty w(x)\theta^x$, it follows from 
\[
(1/g)^\prime(\theta)=\sum_{x=1}^\infty xw(x)\theta^{x-1}>0~~\text{ and }~~(1/g)^{\prime\prime}(\theta)=\sum_{x=2}^\infty x(x-1)w(x)\theta^{x-2}>0
\]
that $g(\cdot)$ is monotonically decreasing on $[0,\theta_*]$ and $(1/g)^\prime(\cdot)$ is monotonically increasing on $[0,\theta_*]$.
Therefore, it follows from 
\[
\log f(x|\theta)=\log(1-\pi)-\log(1/g(\theta))+x\log\theta+\log w(x) 
\]
that 
\begin{eqnarray*}
\frac{{\sf d}(\log f(x|\theta))}{{\sf d}\theta}
=\frac{1}{\theta}\left(x-\theta g(\theta)(1/g)^\prime(\theta)\right)
\geq\frac{1}{\theta}\left(x-\theta_*g(0)(1/g)^\prime(\theta_*)\right)
\geq\frac{1}{\theta}\left(x-C_1\right)
\geq 0
\end{eqnarray*}
for all $x\geq C_1$. 
Therefore we have
\[
h_Q(x)=\int_0^{\theta_*}f(x|\theta)dQ
\leq\sup_{\theta\in[0,\theta_*]}f(x|\theta)
=f(x|\theta_*)
\]
and 
\begin{eqnarray*}
\sum_{x\geq C_1}h_Q(x)\alpha^x
\leq\sum_{x\geq C_1}f(x|\theta_*)\alpha^x
\leq\sum_{x\geq 0}f(x|\theta_*)\alpha^x
\leq\frac{g(\theta_*)}{g(\alpha\theta_*)}
<\infty.
\end{eqnarray*}

For any fixed $k>0$, define $A_n$ to be the event 
\[
A_n:=\Big\{h^{\rm obs}_Q(x)>kh_Q(x)\alpha^{x/3}\text{ for some }x\geq C_1\Big\}.
\]
Then, it follows from Markov's inequality that 
\begin{eqnarray*}
P(A_n)
\leq\sum_{x\geq C_1}P(h^{\rm obs}_Q(x)>kh_Q(x)\alpha^{x/3})
\leq\frac{1}{k}\sum_{x\geq C_1}\E\{h^{\rm obs}_Q(x)\}\frac{1}{h_Q(x)\alpha^{x/3}}
\leq\frac{1}{k}\frac{1}{\alpha^{1/3}-1}.
\end{eqnarray*}
Thus, $P(A_n)$ can be made arbitrarily small by choosing $k$ large enough and on the complement of $A_n$ we have
\[
\sum_{x\geq C_1}\frac{(h^{\rm obs}_Q(x))^2}{h_Q(x)}\alpha^{x/3}\leq k^2\sum_{x\geq C_1}h_Q(x)\alpha^{x}\leq k^2\frac{g(\theta_*)}{g(\alpha\theta_*)}.
\]
Therefore, for an arbitrary $\delta\in(0,1)$, we have 
\[
\sum_{x\geq C_1}\frac{(h^{\rm obs}_Q(x))^2}{h_Q(x)}\alpha^{x/3}\leq \frac{g(\theta_*)}{g(\alpha\theta_*)}\left(\frac{1}{\delta}\frac{1}{\alpha^{1/3}-1}\right)^2
\]
with probability at least $1-\delta$.
Thus, for an arbitrary $\delta\in(0,1)$, with probability at least $1-\delta$, it follows that 
\begin{eqnarray*}
\Big\{\sum_{x\geq C_1}\frac{(h^{\rm obs}_Q(x)-h_Q(x))^2}{h_Q(x)}A^{\frac{x}{1-\gamma}}\Big\}^{1-\gamma}
\leq\Big\{\frac{g(\theta_*)}{g(\alpha\theta_*)}\left(\frac{1}{\delta}\frac{1}{\alpha^{1/3}-1}\right)^2+\frac{g(\theta_*)}{g(\alpha\theta_*)}\Big\}^{1-\gamma}
\leq\frac{C_2}{\delta^{2-2\gamma}},
\end{eqnarray*}
where $C_2=C_2(\theta_*)=g(\theta_*)[1/(\alpha^{1/3}-1)^2+1]/g(\alpha\theta_*)$ is a constant.
For an arbitrary $\delta\in(0,1/2)$, with probability at least $1-2\delta$, it follows that 
\begin{eqnarray*}
n^{1-\epsilon}\sum_{x\geq C_1}\frac{(h^{\rm obs}_Q(x)-h_Q(x))^2}{h_Q(x)}
\leq
n^{1-\epsilon}
\frac{1}{\left(\alpha^{\frac{1-\gamma}{3\gamma}}-1\right)^\gamma}\frac{1}{(n\delta)^\gamma}
\frac{C_2}{\delta^{2-2\gamma}}
=
n^{1-\epsilon-\gamma}
\frac{C_2}{\left(\alpha^{\frac{1-\gamma}{3\gamma}}-1\right)^\gamma}\frac{1}{\delta^{2-\gamma}}.
\end{eqnarray*}
Thus, by letting $\gamma$ go to $1-\epsilon$, we have
\begin{eqnarray*}
\sum_{x\geq C_1}\frac{(h^{\rm obs}_Q(x)-h_Q(x))^2}{h_Q(x)}
&\leq&
\frac{C_2}{\left(\alpha^{\frac{\epsilon}{3(1-\epsilon)}}-1\right)^{1-\epsilon}}\frac{1}{n^{1-\epsilon}}\frac{1}{\delta^{1+\epsilon}}.
\end{eqnarray*}
As a result, for arbitrary $\delta\in(0,1/2)$ and $\epsilon\in(0,1)$, with probability at least $1-2\delta$, we have
\begin{eqnarray*}
\text{KL}(h^{\rm obs},h_Q)
\leq\text{KL}(t_Q^{\rm obs},t_Q)+\sum_{x\geq C_1}\frac{(h^{\rm obs}_Q(x)-h_Q(x))^2}{h_Q(x)}
\leq\text{KL}(t_Q^{\rm obs},t_Q)+C_3\frac{1}{n^{1-\epsilon}}\frac{1}{\delta^{1+\epsilon}},
\end{eqnarray*}
where $C_3=C_3(\epsilon,\theta_*)=C_2/(\alpha^{\frac{\epsilon}{3(1-\epsilon)}}-1)^{1-\epsilon}$.

\vspace{0.2cm}

\noindent\textbf{Step 2(b).} We then upper bound $\text{KL}(t_Q^{\rm obs},t_Q)$. It follows from \cite{mardia2020concentration} that with probability at least $1-\delta$,
\begin{eqnarray*}
\text{KL}(t^{\rm obs},t_Q)\leq \frac{C_1+1}{2n}\log\frac{4n}{C_1+1}+\frac{1}{n}\log\frac{3e}{\delta},
\end{eqnarray*}
and hence for any $\epsilon\in(0,1)$ and $\delta\in(0,1/3)$, with probability at least $1-3\delta$, 
\begin{eqnarray*}
\text{KL}(h^{\rm obs}_Q,h_Q)
\leq\frac{1}{N\delta^{1+\epsilon}}\left(3C_1\log(2n)+C_3n^\epsilon\right).
\end{eqnarray*}
Therefore, it follows that there exists a constant $C_4=C_4(\epsilon,\theta_*)$ such that for any $n\geq 1$
\[
\text{KL}(h^{\rm obs}_Q,h_Q)\leq\frac{C_4}{n^{1-\epsilon}\delta^{1+\epsilon}}
\]
holds with probability at least $1-3\delta$ for any $\epsilon\in(0,1)$ and $\delta\in(0,1/3)$.
Therefore, 
\begin{eqnarray*}
\left|\sum_{x=0}^\infty b_x\left(h^{\rm obs}_Q(x)-h_{\widehat{Q}}(x)\right)\right|
&\leq&\max_{x\geq 0}|b_x|\sqrt{\frac{C_4}{2n^{1-\epsilon}\delta^{1+\epsilon}}}
\end{eqnarray*}
holds for all $n\geq n_1$ with probability at least $1-3\delta$ for any $\epsilon\in(0,1)$ and $\delta\in(0,1/3)$.
\end{proof}

\begin{proof}[Proof of Lemma~\ref{Prop:BoundOnf1}.]
This proof consists of two steps. In the first step, we prove the existence of $\hat{\ell}$ and upper bound the difference between $\hat{\ell}$ and $\ell$.
In the second step, we upper bound coefficients of $\hat{\ell}$, i.e., $\max_{x\geq0}|b_x|$.

\noindent\textbf{Step 1.}
It follows from $\sum_{x=0}^\infty f(x|\theta)=1$ that $\sum_{x=0}^\infty g(\theta)w(x)\theta^x=1$ and hence $g(\theta)>0$ for $\theta\in[0,\theta_*]$.
As a consequence, $1/g(\theta)=\sum_{x=0}^\infty w(x)\theta^x$ on $[0,\theta_*]$.

Since $\theta\mapsto\sum_{x=0}^\infty w(x)\theta^x$ is a continuous function on $[-\theta_*,\theta_*]$ with $w(0)>0$, there exists a universal constant $\theta_0\in(0,\theta_*]$ such that $\theta\mapsto\sum_{x=0}^\infty w(x)\theta^x$ is strictly positive on $[-\theta_0,\theta_*]$.
For $\theta\in[-\theta_0,0)$, define $1/g(\theta):=\sum_{x=0}^\infty w(x)\theta^x$ and $\ell(\theta):=-\ell(-\theta)$.
Then $\theta\mapsto\ell(\theta)$ is a $1$-Lipschitz function on $[-\theta_0,\theta_*]$ and for any $\theta_1,\theta_2\in[-\theta_0,\theta_*]$ we have 
\begin{align*}
|\ell(\theta_1)/g(\theta_1)-\ell(\theta_2)/g(\theta_2)|
&\leq|\ell(\theta_1)/g(\theta_1)-\ell(\theta_2)/g(\theta_1)|+|\ell(\theta_2)/g(\theta_1)-\ell(\theta_2)/g(\theta_2)|\\
&\leq|\theta_1-\theta_2|\{1/g(\theta_*)+\theta_*(1/g)^\prime(\theta_*)\}.
\end{align*}
Therefore, it follows from Jackson's theorem (see Lemma~\ref{lem:jackson}) that there exists a polynomial $\sum_{x=0}^k v_x\theta^x$ of degree $k\geq 1$ such that 
\[
\sup_{\theta\in[-\theta_0,\theta_*]}|\ell(\theta)/g(\theta)-\sum_{x=0}^k v(x)\theta^x|\leq C_1/k, 
\]
where $C_1=C_1(\theta_*)$ is a positive constant independent of $k$ and $\ell$ and $v_x\in\mathbb{R}$ for all $x=0,\ldots,k$.
Let $b_x=v_x/\{w(x)(1-\pi)\}$ for $x=1,\ldots,k$ and $b_x=0$ for $x=0$.
Then it follows from 
\[
|v_0|\leq C_1/k+|\ell(0)/g(0)|=C_1/k 
\]
that 
\[
\sup_{\theta\in[-\theta_0,\theta_*]}\Big|\frac{\ell(\theta)}{g(\theta)}-b_0\frac{\pi+(1-\pi)g(\theta)w(0)}{g(\theta)}-(1-\pi)\sum_{x=1}^k b_xw(x)\theta^x\Big|\leq 2C_1/k,
\]
and hence 
\[
\sup_{\theta\in[-\theta_0,\theta_*]}\Big|\ell(\theta)-\hat\ell(\theta)\Big|
\leq C_2/k,
\]
where 
\[
\hat\ell(\theta):=b_0\pi+(1-\pi)g(\theta)\sum_{x=0}^k b_xw(x)\theta^x=\sum_{x=0}^kb_xf(x|\theta) 
\]
and $C_2=C_2(\theta_*)$ is a positive constant independent of $k$ and $\ell$.

\vspace{0.2cm}

\noindent\textbf{Step 2}.
To bound the coefficients $b_x$'s, we first define a polynomial 
\[
\theta\mapsto r(\theta):=\sum_{x=0}^{k}v_x(\theta_0\theta)^x \text{ on } [-1,1] 
\]
and note that 
\[
\sup_{\theta\in[-1,1]}|r(\theta)|\leq C_1/k+\sup_{\theta\in[-\theta_0,\theta_0]}|\ell(\theta)/g(\theta)|\leq C_1/k+\theta_0/g(\theta_0).
\]
We then apply Lemma~\ref{fact:coeff-bound} on the polynomial $r(\theta)$, and it follows that
\[
|v(x)|\theta_0^x\leq \max_{|\theta|\leq 1}|r(\theta)| \cdot k^x/x!
\leq C_3k^x/x!,
\]
where $C_3=C_3(\theta_*)$ is a positive constant.
Hence 
\[
|b_x|=|v_x|/\{w(x)(1-\pi)\}\leq C_3\cdot(k/\theta_0)^x/\{x!w(x)(1-\pi)\}
\] 
and 
\[
\max_{x\in[0,k]} |b_x|
\leq \frac{C_3}{1-\pi}\cdot\max_{x\in[0,k]}\frac{(k/\theta_0)^x}{x!w(x)}
\leq C_4\cdot\frac{(k/\theta_0)^k}{k!}\cdot\max_{1\leq x\leq k}1/w(x)
\leq C_4\cdot(e/\theta_0)^k\cdot\max_{1\leq x\leq k}1/w(x),
\]
where $C_4=C_4(\theta_*)$ is a positive constant.
It follows from \citet[Corollary 1.1.10]{krantz2002primer} that $w(x)\leq C_5/\theta_*^x$ for all $x\in\mathbb{N}$ and some universal constant $C_5\geq1$ and hence
\[
(e/\theta_0)^k/w(k)\geq (\theta_*e/\theta_0)^k/C_5\geq e^k/C_5>1
\]
for all sufficiently large $k$.
\end{proof}

\bibliographystyle{apalike}
\bibliography{smooth}

\end{document}